\documentclass[lettersize,journal]{IEEEtran}
\usepackage{amsmath,amsfonts}
\usepackage{amsthm}
\usepackage{algorithmic}
\usepackage{algorithm}
\usepackage{array}
\usepackage[caption=false,font=normalsize,labelfont=sf,textfont=sf]{subfig}
\usepackage{graphicx}
\usepackage{cite}
\usepackage{colortbl}
\usepackage{amssymb}

\newtheorem{theorem}{Theorem}

\newtheorem{corollary}{Corollary}
\newtheorem{remark}{Remark}
\newtheorem{definition}{Definition}
\newtheorem{assumption}{Assumption}
\newtheorem{lemma}{Lemma}
\newtheorem{example}{Example}
\begin{document}
\title{Weak Closed-loop Solvability for Discrete-time Linear-Quadratic Optimal Control}

\author{ Yue Sun$^a$, Xianping Wu$^{b,*}$\thanks{$^*$\textit{corresponding author. Email:xianpingwu@gdut.edu.cn.}}
	 and Xun Li$^c$ \\
		$^a$\textit{School of Control Science and Engineering, Shandong University, Jinan, Shandong, China}\\
		$^b$\textit{School of  Mathematics and Statistics, Guangdong University of Technology, Guangzhou, China}\\
        $^c$\textit{Department of Applied Mathematics, The Hong Kong Polytechnic University, Hong Kong, China} 
 }



\maketitle

\begin{abstract}
In this paper, the open-loop, closed-loop, and weak closed-loop solvability for discrete-time linear-quadratic (LQ) control problem is considered due to the fact that it is always open-loop optimal solvable if the LQ control problem is closed-loop optimal solvable but not vice versa. The contributions are two-fold. On the one hand, the equivalent relationship between the closed-loop optimal solvability and the solution of the generalized Riccati equation is given. On the other hand, when the system is merely open-loop solvable, we have found the equivalent existence form of the optimal solution by perturbation method, which is said to be a weak closed-loop solution. Moreover, it obtains that there is an open-loop optimal control with a linear feedback form of the state. The essential technique is to solve the forward and backward difference equations by iteration. An example sheds light on the theoretical results established.
\end{abstract}


\section{Introduction}

Optimal control theory is an essential component of modern control theory, and linear-quadratic (LQ) optimal control is the fundamental problem and the core foundation of optimal control. The research on traditional LQ control problems began in the 1960s, and after Kalman’s pioneering work \cite{Kalman1960} and  \cite{Kalman1963}, LQ control problems have been widely and deeply studied. Deterministic LQ problem, such as the study in \cite{Naidu} with single input, is concerned with
\begin{eqnarray*}
  &&\min \sum^{N-1}_{t=0}(x'_tQx_t+2u'_tSx_t+u'_tRu_t)+x'_{N}Hx_{N},\\
  &&\mbox{subject\!\quad\!to}\quad  x_{t+1}=Ax_t+Bu_t,
\end{eqnarray*}
where $x_t\in \mathbb{R}^n$ is the state with the initial value $x_0$, $u_t\in \mathbb{R}^m$ is the control input, $A$ and $B$ are constant matrices with compatible dimensions.
Under the traditional assumption, the cost functional always has positive semi-definite weighting matrices for the control and the state. For a constant $\delta >0$,
assume that
\begin{eqnarray*}
  R\geq \delta I, \quad Q-S'R^{-1}S\geq 0.
\end{eqnarray*}
Under these conditions, the LQ problem admits a unique solution if and only if the Riccati equation with the terminal value $P_{N}=H$ has a unique positive definite solution satisfying
\begin{eqnarray*}\label{0}
P_t&\hspace{-0.8em}=&\hspace{-0.8em} Q+A'P_{t+1}A-A'P_{t+1}B(R+B'P_{t+1}B)^{-1} B'P_{t+1}A,
\end{eqnarray*}
while the unique open-loop optimal solution to the problem is a linear feedback form of the state.
Based on the assumption of positive semi-definite weighting matrices, the deterministic LQ problem with time delay has been considered since the 1970s. Koivo-Lee \cite{Koivo1972}, Alekal-Brunovsky-Chyung-Lee \cite{Alekal1971} and Artstein \cite{Artstein1982}, Delfour \cite{Delfour1986} successively focused on the single-delay case and the multiple-delay case. In the positive semi-definite case, no problems exist with the existence of an optimal solution since the optimal cost functional is always positive semi-definite. For the indefinite LQ problem, that is, there is only an assumption on the symmetry of the weighting matrices, referring to Ran-Trentelman \cite{Ran1993}, which studied the infinite-horizon LQ problem with an indefinite quadratic form and presented necessary and sufficient conditions for the existence of an optimal control. As pointed out in Ferrante-Ntogramatzidis \cite{Ferrante2015}, unlike the positive semi-definite case, the finite-horizon indefinite LQ problem may not admit solutions because the performance index can be arbitrarily negative through suitable control actions. 
Using the generalized Riccati difference equation:
\begin{eqnarray*}\label{00}
P_t&\hspace{-0.8em}=&\hspace{-0.8em} Q+A'P_{t+1}A-A'P_{t+1}B(R+B'P_{t+1}B)^{\dagger} B'P_{t+1}A,
\end{eqnarray*}
with the terminal value $P_{N}=H$, Ferrante-Ntogramatzidis \cite{Ferrante2015} presented
a necessary and sufficient condition for the existence of an optimal control.

In the traditional LQ optimal control with a single input, the open-loop optimal solvability is equivalent to the closed-loop optimal solvability. In the indefinite LQ problem,  the closed-loop optimal solvability implies the open-loop optimal solvability, but we have found in the following example that it's not vice versa.
\begin{example}\label{example1}
Consider the following system with one-dimension state equation:
\begin{eqnarray}\label{01}
  x_{t+1}= x_t+u_t, \quad t\in\{0, 1\},
\end{eqnarray}
with the initial value $x_0$. The associated cost functional is given by
\begin{eqnarray}\label{02}
J(x_0, u) = x^2_2-\sum\limits^{1}_{t=0}u^2_t.
\end{eqnarray}

Following from the difference Riccati equation (\ref{00})
with the terminal value $P_2=1$, and by simple calculation, the solution is exactly $P_t\equiv 1$, $t\in \{0, 1, 2\}$.
Combining with the maximum principle, a generalized Riccati equation approach specifies the corresponding  closed-loop solution as follows
\begin{eqnarray}\label{03}
  u^*_t = -(R+B'P_{t+1}B)^{\dagger}B'P_{t+1}Ax_t \equiv0, \quad t\in\{0, 1\},
\end{eqnarray}
due to $R=-1$ and $0^{\dagger}=0$.
It should be noted that (\ref{03}) is no longer an optimal open-loop solution if $x_0\neq 0$. To be specific, adding $u^*_t$ in (\ref{03}) into (\ref{01}), the state $x^*_t$ is calculated as
\begin{eqnarray*}
  x^*_t = x_0 +\sum\limits^{t-1}_{i=0}u^*_i=x_0, \quad t\in\{1, 2\}.
\end{eqnarray*}
To this end, the cost functional in (\ref{02}) is given by
\begin{eqnarray*}
  J(x_0, u^*) =x^2_0 >0,
\end{eqnarray*}
for any $x_0\neq 0$.

Moreover, let $\bar{u}_t$ be the control defined by
\begin{eqnarray*}
 \bar{u}_t=Mx_t, \quad M=-0.3522.
\end{eqnarray*}
In this way, we have
\begin{eqnarray*}
\bar{u}_0&\hspace{-0.8em}=&\hspace{-0.8em}Mx_0,\\
 \bar{u}_1&\hspace{-0.8em}=&\hspace{-0.8em}Mx_1=M(x_0+u_0)=M(x_0+Mx_0)\nonumber\\
 &\hspace{-0.8em}=&\hspace{-0.8em}M(1+M)x_0.
\end{eqnarray*}
Substituting $\bar{u}_0$ and $\bar{u}_1$ into (\ref{01}) yields,
\begin{eqnarray*}
  x_2 =x_0+\bar{u}_0+\bar{u}_1.
\end{eqnarray*}
Thus, the cost functional $J(x_0, \bar{u})$ is given by
\begin{eqnarray}\label{010}
J(x_0, \bar{u})&\hspace{-0.8em}=&\hspace{-0.8em}x^2_2-\bar{u}^2_0-\bar{u}^2_1\nonumber\\
&\hspace{-0.8em}=&\hspace{-0.8em}(x_0+\bar{u}_0+\bar{u}_1)^2-\bar{u}^2_0-\bar{u}^2_1\nonumber\\
&\hspace{-0.8em}=&\hspace{-0.8em}x^2_0+2x_0\bar{u}_0+2x_0\bar{u}_1+2\bar{u}_0\bar{u}_1\nonumber\\
&\hspace{-0.8em}=&\hspace{-0.8em}(2M^3+4M^2+4M+1)x^2_0.
\end{eqnarray}
Since the real root of $2M^3+4M^2+4M+1=0$ is expressed as
\begin{eqnarray*}
  M&\hspace{-0.8em}=&\hspace{-0.8em}\sqrt[3]{\frac{13}{108}+\sqrt{\frac{13^2}{108^2}+\frac{2^3}{9^3}}}
  + \sqrt[3]{\frac{13}{108}-\sqrt{\frac{13^2}{108^2}+\frac{2^3}{9^3}}}\nonumber\\
  &\hspace{-0.8em}&\hspace{-0.8em}-\frac{2}{3}=-0.3522,
\end{eqnarray*}
using $\bar{u}_t=Mx_t$ and $M=-0.3522$ for any $x_0 \neq 0$,
we get
\begin{eqnarray*}
J(x_0, \bar{u})&\hspace{-0.8em}=&\hspace{-0.8em}0\cdot x^2_0=0<J(x_0, u^*).
\end{eqnarray*}
Based on the discussion above, since the cost functional is nonnegative, $\bar{u}$ is open-loop optimal for the initial value $x_0$, but $u^*$ is not.
\end{example}

The above example shows that the generalized Riccati equation (\ref{00}) is helpless in the open-loop solvability of certain LQ problems. Does it inspire that when the LQ problem is merely open-loop solvable, is it possible to get a linear feedback form of state for an open-loop optimal control? Motivated by it, combined with the study of open-loop and closed-loop solvability of It\^{o} differential system in \cite{Sun2014} and \cite{Sun2016}, the necessity and sufficient condition for the closed-loop solvability of LQ problem is derived for a system with discrete time, which is the fundamental of the following discussion. Moreover, by using the perturbation approach introduced in \cite{Sun2018} and \cite{Wen2021}, some equivalent conditions for the open-loop solvability of the LQ problem are given. Moreover, a weak closed-loop solvability of the LQ problem is introduced to obtain a linear state presentation of the open-loop optimal control. And it proves that the existence of a weak closed-loop solution is equivalent to the open-loop solvability of the LQ problem, whose outcome is an open-loop optimal control.

The outline of this paper is arranged as follows. Problem formulation and some preliminaries related to the LQ control problem are introduced in Section II. The perturbation method investigates the equivalent conditions of open-loop solvability in Section III. The relationship between open-loop and weak closed-loop solvability is proposed in Section IV. A numerical example in Section V sheds light on the results established in this paper.

\emph{Notations:} $\mathbb{R}^n$ denotes an $n$-dimensional Euclidean space.
$A'$ means the transpose of the matrix $A$.
A symmetric matrix $A >0$ (reps. $\geq0$) means that it is positive definite (reps. positive semi-definite).
Let $\mathbb{N}$ and $\tilde{\mathbb{N}}$ represent the sets of $\{0, 1, \cdots, N-1\}$ and $\{0, 1, \cdots, m\}$ with $m\leq N-2$, respectively.
Denote $L^2(\mathbb{N}; \mathbb{R}^n)=\{\phi: \mathbb{N}\rightarrow \mathbb{R}^n\Big| \sum\limits^{N-1}_{t=0} |\phi_t|^2< \infty\}$.
Let $M^{\dag}$ be  Moore-Penrose inverse of matrix $M$ satisfying $MM^{\dag}M=M$, $M^{\dag}MM^{\dag}=M^{\dag}$, $(MM^{\dag})'=MM^{\dag}$ and $(M^{\dag}M)'=M^{\dag}M$.
$Rang(A)=\{Ax| x\in \mathbb{R}^n\}$ with $n$ being the dimension of $x$ is the range of $A$.

\section{Problem Formulation}
Consider the discrete-time system:
\begin{eqnarray}\label{1}
  x_{t+1} &\hspace{-0.8em}=&\hspace{-0.8em} A_tx_t+B_tu_t,
\end{eqnarray}
where $x_t\in \mathbb{R}^n$ is the state with the initial value $x_0${\color{blue},} $u_t\in \mathbb{R}^m$ is the control input{\color{blue},} $A_t$ and $B_t$ are deterministic matrix-valued functions of compatible dimensions. The associated cost functional is given by
\begin{eqnarray}\label{2}
J(x_0, u)
&\hspace{-0.8em}=&\hspace{-0.8em} \sum\limits^{N-1}_{t=0}[x'_tQ_tx_t+2u'_tS_tx_t+u'_tR_tu_t] \nonumber \\
&\hspace{-0.8em} &\hspace{-0.8em}  +x'_{N}Hx_{N},
\end{eqnarray}
where $H\in \mathbb{R}^{n\times n}$ is a symmetric constant matrix, $S_t$ is  deterministic matrix-valued function, $Q_t$ and $R_t$ are   deterministic and  symmetric matrix-valued functions.

The problem considered in this paper is given below.
\smallskip

\emph{Problem (LQ).} For any given initial value $x_0\in \mathbb{R}^n$, find a controller
$\bar{u}\in L^2(\mathbb{N}; \mathbb{R}^{m})$ satisfying
\begin{eqnarray*}
J(x_0, \bar{u}) \leq J(x_0, u), \quad\forall u\in L^2(\mathbb{N}; \mathbb{R}^{m}).
\end{eqnarray*}

Moreover, the value function of Problem (LQ) with the initial value $x_0$ is defined as
\begin{eqnarray*}
  V(x_0) =\inf_{u\in L^2(\mathbb{N}; \mathbb{R}^{m})}J(x_0, u).
\end{eqnarray*}

Some assumptions regarding the coefficients are established.
\begin{assumption}\label{A1}
The coefficients $A: \mathbb{N}\rightarrow \mathbb{R}^{n\times n}$ and
 $B: \mathbb{N}\rightarrow \mathbb{R}^{n\times m}$ are uniformly bounded mappings.
\end{assumption}

\begin{assumption}\label{A2}
The coefficients $Q: \mathbb{N}\rightarrow \mathbb{R}^{n\times n}$,
$R: \mathbb{N}\rightarrow \mathbb{R}^{m\times m}$ and
$S: \mathbb{N}\rightarrow \mathbb{R}^{m\times n}$ are uniformly bounded mappings.
\end{assumption}
%

Before proposing the main results of this paper, the definitions of open-loop, closed-loop, and weak closed-loop solvable Problem (LQ) are given.

\begin{definition}
Problem (LQ) is (uniquely) open-loop solvable with the initial value $x_0$ if there exists a (unique) $\bar{u}=\bar{u}(x_0)\in L^2(\mathbb{N}; \mathbb{R}^{m})$, which means it's related with the initial value $x_0$, satisfying
\begin{eqnarray*}
J(x_0, \bar{u}) \leq J(x_0, u), \quad\forall u\in L^2(\mathbb{N}; \mathbb{R}^{m}),
\end{eqnarray*}
then, $\bar{u}$ is called as the open-loop optimal control with the initial value $x_0$;
Problem (LQ) is (uniquely) open-loop solvable if it is (uniquely) open-loop solvable with any initial value $x_0$.
\end{definition}

\begin{definition}\label{D2}
Let $K\in \mathbb{R}^{m\times n}$ and $v\in \mathbb{R}^{m}$ unrelated with the initial value $x_0$. The pair$(K, v)$ is called as the closed-loop solution on $\mathbb{N}$ if $\sum\limits^{N-1}_{t=0} |K_t|^2< \infty$ and $\sum\limits^{N-1}_{t=0} |v_t|^2< \infty$, and denote
$\mathcal{K}_{\mathbb{N}}$
as the set of all closed-loop solutions;
$(K^*, v^*)\in \mathcal{K}_{\mathbb{N}}$ is called as the closed-loop optimal solution on $\mathbb{N}$ if
$\forall (K, v)\in\mathcal{K}_{\mathbb{N}}$, there holds
\begin{eqnarray*}
J(x_0, K^*x^*+v^*) \leq J(x_0, Kx+v),
\end{eqnarray*}
where $x^*$ and $x$ are the solution to the corresponding closed-loop systems satisfying
\begin{eqnarray*}
x_{t+1} &\hspace{-0.8em}=&\hspace{-0.8em} (A_t+B_tK_t)x_t+B_tv_t,\\
x^*_{t+1} &\hspace{-0.8em}=&\hspace{-0.8em} (A_t+B_tK^*_t)x^*_t+B_tv^*_t,
\end{eqnarray*}
respectively, with the initial value $x_0$ and $x^*_0=x_0$;
If there exists a (unique) closed-loop optimal solution on $\mathbb{N}$, then Problem (LQ) is called as (unique) closed-loop solvable.
\end{definition}

\begin{remark}
Following from the Definition \ref{D2}, $\forall x_0\in \mathbb{R}^n$, $(K^*, v^*)$ is an closed-loop optimal solution  on $\mathbb{N}$ to Problem (LQ) if and only if
\begin{eqnarray*}
J(x_0, K^*x^*+v^*) \leq J(x_0,u), \quad \forall u\in L^2(\mathbb{N}; \mathbb{R}^{m}).
\end{eqnarray*}
Hence, it should be noted that  $u^*=K^*x^*+v^*$  is also an open-loop optimal solution to the Problem (LQ) with the initial value $x_0^*$. That is to say, the existence of the closed-loop optimal solution to Problem (LQ) indicates the existence of the open-loop optimal solution to it, but the reverse is not true, which can be seen in Example \ref{example1}. When Problem (LQ) is merely open-loop solvable, whether there also exists a control with a linear feedback form of the state is studied in this paper. In the following, a weak closed-loop solvability of Problem (LQ) is defined.
\end{remark}

\begin{definition}\label{definition3}
For $\tilde{\mathbb{N}}\subseteq {\mathbb{N}}$,
let $K\in L^2(\tilde{\mathbb{N}}; \mathbb{R}^{m\times n})$ and
$v\in L^2(\tilde{\mathbb{N}}; \mathbb{R}^{m})$. $(K, v)$ is called as  a weak closed-loop solution on $\tilde{\mathbb{N}}$ for any initial value $x_0$, if $u=Kx
+v \in L^2(\mathbb{N}; \mathbb{R}^{m})$, where $x$ is the solution to the weak closed-loop state
\begin{eqnarray*}
x_{t+1} &\hspace{-0.8em}=&\hspace{-0.8em} (A_t+B_tK_t)x_t+B_tv_t,
\end{eqnarray*}
with the initial value $x_0$, and the set of all weak closed-loop solutions is denoted as $\mathcal{K}^w_{\mathbb{N}}$;
$(K^*, v^*)\in \mathcal{K}^w_{\mathbb{N}}$ is said to be weak closed-loop optimal, if for any initial value $x_0$ and $ (K, v)\in\mathcal{K}^w_{\mathbb{N}}$, there holds
\begin{eqnarray*}
J(x_0, K^*x^*+v^*) \leq J(x_0, Kx+v),
\end{eqnarray*}
where $x^*$ is the solution to the weak closed-loop state
\begin{eqnarray*}
x^*_{t+1} &\hspace{-0.8em}=&\hspace{-0.8em} (A_t+B_tK^*_t)x^*_t+B_tv^*_t,
\end{eqnarray*}
with the initial value $x^*_0=x_0$;
Problem (LQ) is said to be (uniquely) weakly closed-loop optimal solvable if there exists a  (uniquely) weak closed-loop optimal solution $(K^*, v^*)$ on $\tilde{\mathbb{N}}$.
\end{definition}

\begin{remark}
Following from the Definition \ref{definition3}, for any $x_0\in \mathbb{R}^n$, $(K^*, v^*)\in\mathcal{K}^w_{\mathbb{N}}$ is a weak closed-loop optimal solution to Problem (LQ) if and only if
\begin{eqnarray*}
J(x_0, K^*x^*+v^*) \leq J(x_0,u), \quad \forall u\in L^2(\mathbb{N}; \mathbb{R}^{m}).
\end{eqnarray*}
\end{remark}

Based on the discussion above, the conditions of open-loop and closed-loop solvable for the Problem (LQ) are derived.

\begin{lemma}\label{lemma1}
Under Assumption \ref{A1}-\ref{A2}, if Problem (LQ) is open-loop solvable, then $u$ subject to (\ref{1}) satisfies the equilibrium condition
\begin{eqnarray}\label{3}
  0 &\hspace{-0.8em}=&\hspace{-0.8em} R_tu_t+S_tx_t+B'_t\lambda_t,
\end{eqnarray}
where the co-state satisfies the following backward equation with the terminal value $\lambda_{N-1}=Hx_{N}$ and
\begin{eqnarray}\label{4}
  \lambda_{t-1}&\hspace{-0.8em}=&\hspace{-0.8em}Q_tx_t+S'_tu_t+A'_t\lambda_t.
\end{eqnarray}
\end{lemma}
\begin{proof}
The proof can be obtained immediately from \cite{Zhang2012}.
\end{proof}

\begin{lemma}\label{lemma2}
Under Assumption \ref{A1}-\ref{A2}, let $(K^*, v^*)$ be a closed-loop solution to Problem (LQ). Then $\forall x_0\in \mathbb{R}^n$, the following FBDEs admits an solution ($x^*_t, \lambda^*_{t-1}$) for $t\in \mathbb{N}$:
\begin{eqnarray}
x^*_{t+1} &\hspace{-0.8em}=&\hspace{-0.8em} (A_t+B_tK^*_t)x^*_t+B_tv^*_t,\label{5-1}\\
\lambda^*_{t-1}&\hspace{-0.8em}=&\hspace{-0.8em}(Q_t+S'_tK^*_t)x^*_t+S'_tv^*_t+A'_t\lambda^*_t,
\label{5-2}
\end{eqnarray}
with the initial value $x^*_0=x_0$ and $\lambda^*_{N-1}=Hx^*_{N}$. Moreover, the equilibrium strategy satisfies
\begin{eqnarray}\label{6}
  0 &\hspace{-0.8em}=&\hspace{-0.8em} (R_tK^*_t+S_t)x^*_t+R_tv^*_t+B'_t\lambda^*_t.
\end{eqnarray}
\end{lemma}

\begin{proof}
If $(K^*, v^*)$ is a closed-loop solution to Problem (LQ), reconsidering (\ref{1})-(\ref{2}), one has
\begin{eqnarray*}
x_{t+1} &\hspace{-0.8em}=&\hspace{-0.8em} (A_t+B_tK^*_t)x_t+B_tv^*_t,\\
 J(x_0, K^*x^*+v^*) &\hspace{-0.8em}=&\hspace{-0.8em} \sum\limits^{N-1}_{t=0}\{x'_t[Q_t+(K^*_t)'R_tK^*_t+(K^*_t)'S_t\nonumber\\
 &\hspace{-0.8em} &\hspace{-0.8em} +S'_tK^*_t]x_t
 +2(v^*_t)'(S_t+R_tK^*_t)x_t\nonumber\\
 &\hspace{-0.8em}&\hspace{-0.8em}+(v^*_t)'R_tv^*_t\}+x'_{N}Hx_{N}.
\end{eqnarray*}
Repeating the process of Lemma \ref{lemma1}, (\ref{5-2})-(\ref{6}) are derived accordingly.
\end{proof}

The equivalent relationship between the closed-loop solvable of the Problem (LQ) and the solvability of the generalized Riccati equation is given below. 

\begin{theorem}\label{theorem1}
Under Assumption \ref{A1}-\ref{A2}, Problem (LQ) has a closed-loop solution $(K^*, v^*)$ if and only if the  generalized Riccati equation
with the terminal value $P_{N}=H$
\begin{eqnarray}\label{7}
P_t&\hspace{-0.8em}=&\hspace{-0.8em} Q_t+A'_tP_{t+1}A_t-(A'_tP_{t+1}B_t+S'_t)\nonumber\\
&\hspace{-0.8em} &\hspace{-0.8em} \times \hat{R}_t^{\dagger}(B'_tP_{t+1}A_t+S_t),
\end{eqnarray}
admits a solution $P_t$ on $t\in \mathbb{N}$ such that
\begin{eqnarray}
&\hspace{-0.8em}&\hspace{-0.8em}\hat{R}_t \geq 0,\label{8-1}\\
&\hspace{-0.8em}&\hspace{-0.8em}\hat{K}_t=- \hat{R}_t^{\dagger}(B'_tP_{t+1}A_t+S_t)\in L^2(\mathbb{N}; \mathbb{R}^{m\times n}),\label{8-2}\\
&\hspace{-0.8em}&\hspace{-0.8em}Range(B'_tP_{t+1}A_t+S_t)  \subseteq  Range( \hat{R}_t),\label{8-3}
\end{eqnarray}
where $\hat{R}_t=R_t+B'_tP_{t+1}B_t$.
In this case, the closed-loop solution $(K^*, v^*)$ is given by
\begin{eqnarray}
K^*_t&\hspace{-0.8em}=&\hspace{-0.8em} \hat{K}_t+(I-\hat{R}^{\dagger}_t\hat{R}_t) z_t,\label{9-1}\\
v^*_t
&\hspace{-0.8em}=&\hspace{-0.8em}(I-\hat{R}^{\dagger}_t\hat{R}_t)y_t,\label{9-3}
\end{eqnarray}
where $z\in L^2(\mathbb{N}; \mathbb{R}^{m\times n})$ and $y\in L^2(\mathbb{N}; \mathbb{R}^{m})$
are arbitrary.\\
Moreover, the value function is
\begin{eqnarray}\label{10}
  V(x_0)=x'_0P_0x_0.
\end{eqnarray}
\end{theorem}

\begin{proof}
``Necessary.'' Suppose $(K^*, v^*)$ is the optimal closed-loop solution of the Problem (LQ),
according to Lemma \ref{lemma2},  there exist the FBDEs (\ref{5-1})-(\ref{6}) for any $x_0\in \mathbb{R}^n$.
Let ($\bar{x}_t, \bar{\lambda}_{t-1}$) for $t\in\mathbb{N}$ be the solution to the FBDEs
\begin{eqnarray}
\bar{x}_{t+1} &\hspace{-0.8em}=&\hspace{-0.8em} (A_t+B_tK^*_t)\bar{x}_t, \quad \bar{x}_0=x_0,\label{11-1}\\
\bar{\lambda}_{t-1}&\hspace{-0.8em}=&\hspace{-0.8em}(Q_t+S'_tK^*_t)\bar{x}_t+A'_t\bar{\lambda}_t,\quad \bar{\lambda}_{N-1}=H\bar{x}_{N},\label{11-2}
\end{eqnarray}
then, there exists the following equilibrium condition
\begin{eqnarray}\label{112}
  0 &\hspace{-0.8em}=&\hspace{-0.8em} (R_tK^*_t+S_t)\bar{x}_t+B'_t\bar{\lambda}_t.
\end{eqnarray}

Solving (\ref{11-1})-(\ref{11-2}) by backward iteration from $k=N$, one has
\begin{eqnarray*}
\bar{\lambda}_{t-1}&\hspace{-0.8em}=&\hspace{-0.8em}(Q_t+S'_tK^*_t)\bar{x}_t+A'_tP_{t+1}(A_t+B_tK^*_t)\bar{x}_t\nonumber\\
&\hspace{-0.8em}=&\hspace{-0.8em}[Q_t+A'_tP_{t+1}A_t+(A'_tP_{t+1}B_t+S'_t)K^*_t]\bar{x}_t\nonumber\\
&\hspace{-0.8em}\doteq&\hspace{-0.8em} P_t\bar{x}_t,
\end{eqnarray*}
with the terminal value $P_{N}=H$, i.e.,
\begin{eqnarray}\label{12-1}
  P_t &\hspace{-0.8em}=&\hspace{-0.8em} Q_t+A'_tP_{t+1}A_t+(A'_tP_{t+1}B_t+S'_t)K^*_t.
\end{eqnarray}
Adding $\bar{\lambda}_{k}=P_{t+1}\bar{x}_{t+1}$ into (\ref{112}), it yields that
\begin{eqnarray*}
  0 &\hspace{-0.8em}=&\hspace{-0.8em} (R_tK^*_t+S_t)\bar{x}_t+B'_tP_{t+1}(A_t+B_tK^*_t)\bar{x}_t\nonumber\\
  &\hspace{-0.8em}=&\hspace{-0.8em}(\hat{R}_tK^*_t+B'_tP_{t+1}A_t+S_t)\bar{x}_t.
\end{eqnarray*}
The arbitrariness of $\bar{x}_t$ means that
\begin{eqnarray}\label{13}
\hat{R}_tK^*_t+B'_tP_{t+1}A_t+S_t=0,
\end{eqnarray}
which implies $Range(B'_tP_{t+1}A_t+S_t)  \subseteq  Range( \hat{R}_t)$.
To this end, (\ref{12-1}) can be rewritten as
\begin{eqnarray}\label{12-2}
P_t&\hspace{-0.8em}=&\hspace{-0.8em}Q_t+A'_tP_{t+1}A_t+(A'_tP_{t+1}B_t+S'_t)K^*_t
+(K^*_t)'\nonumber\\
&\hspace{-0.8em}&\hspace{-0.8em}\times (\hat{R}_tK^*_t+B'_tP_{t+1}A_t+S_t)\nonumber\\
&\hspace{-0.8em}=&\hspace{-0.8em}Q_t+(A_t+B_tK^*_t)'P_{t+1}(A_t+B_tK^*_t)+S'_tK^*_t\nonumber\\
&\hspace{-0.8em}&\hspace{-0.8em}+(K^*_t)'S_t+(K^*_t)'R_tK^*_t.
\end{eqnarray}

Following from (\ref{13}), one has
\begin{eqnarray*}
 \hat{R}_t^{\dagger}(B'_tP_{t+1}A_t+S_t)=-\hat{R}_t^{\dagger} (\hat{R}_t)K^*_t,
\end{eqnarray*}
that is to say
\begin{eqnarray*}
K^*_t&\hspace{-0.8em}=&\hspace{-0.8em} - \hat{R}_t^{\dagger}(B'_tP_{t+1}A_t+S_t)
+(I-\hat{R}_t^{\dagger}\hat{R}_t) z_t.
\end{eqnarray*}

Consequently,
\begin{eqnarray*}
&\hspace{-0.8em}&\hspace{-0.8em}(A'_tP_{t+1}B_t+S'_t)K^*_t\nonumber\\
=&\hspace{-0.8em}&\hspace{-0.8em} -(K^*_t)'\hat{R}_tK^*_t \nonumber\\
=&\hspace{-0.8em}&\hspace{-0.8em}(K^*_t)'\hat{R}_t\hat{R}_t^{\dagger}(B'_tP_{t+1}A_t+S_t)
\nonumber\\
=&\hspace{-0.8em}&\hspace{-0.8em}-(A'_tP_{t+1}B_t+S'_t)
\hat{R}_t^{\dagger}(B'_tP_{t+1}A_t+S_t).
\end{eqnarray*}
Substituting the above equation to (\ref{12-1}), the Riccati equation (\ref{7}) is obtained.

Next, we will prove $v^*$ satisfying (\ref{9-3}). Denote $\eta_{t-1}=\lambda^*_{t-1}-P_tx^*_t$ for $t\in\mathbb{N}$ with the terminal value $\eta_{N-1}=0$. Adding it into (\ref{6}), it derives
\begin{eqnarray}\label{13-1}
  0 &\hspace{-0.8em}=&\hspace{-0.8em} (R_tK^*_t+S_t)x^*_t+R_tv^*_t+B'_t(\eta_{k}+P_{t+1}x^*_{t+1})\nonumber\\
  &\hspace{-0.8em}=&\hspace{-0.8em}\hat{R}_tv^*_t+(B'_tP_{t+1}A_t+S_t+\hat{R}_tK^*_t)x^*_t+B'_t\eta_{t}\nonumber\\
  &\hspace{-0.8em}=&\hspace{-0.8em}\hat{R}_tv^*_t+B'_t\eta_{t},
\end{eqnarray}
i.e., $v^*_t=-\hat{R}_t^{\dagger}B'_t\eta_{t}
+(I-\hat{R}_t^{\dagger}\hat{R}_t)y_t$,
where $y\in L^2(\mathbb{N}; \mathbb{R}^m)$ is arbitrary. And
\begin{eqnarray*}
&\hspace{-0.8em}&\hspace{-0.8em}(A'_tP_{t+1}B_t+S'_t)v^*_t\nonumber\\
=&\hspace{-0.8em}&\hspace{-0.8em}-(A'_tP_{t+1}B_t+S'_t)\hat{R}_t^{\dagger}B'_t
\eta_{t}
  -(K^*_t)'\hat{R}_t (I-\hat{R}_t^{\dagger}\hat{R}_t)y_t\nonumber\\
=&\hspace{-0.8em}&\hspace{-0.8em}-(A'_tP_{t+1}B_t+S'_t)\hat{R}_t^{\dagger}B'_t\eta_{t}.
\end{eqnarray*}

Combining with (\ref{5-2}), $\eta_{t-1}$ is calculated as
\begin{eqnarray}\label{15}
  \eta_{t-1} &\hspace{-0.8em}=&\hspace{-0.8em} (Q_t+S'_tK^*_t)x^*_t+S'_tv^*_t+A'_t(\eta_t+P_{t+1}x^*_{t+1})\nonumber\\
&\hspace{-0.8em}&\hspace{-0.8em}-P_tx^*_t \nonumber\\
&\hspace{-0.8em}=&\hspace{-0.8em}A'_t\eta_t+(A'_tP_{t+1}B_t+S'_t)v^*_t+(Q_t+A'_tP_{t+1}A_t\nonumber\\
&\hspace{-0.8em}&\hspace{-0.8em}+(A'_tP_{t+1}B_t+S'_t)K^*_t-P_t)x^*_t\nonumber\\
&\hspace{-0.8em}=&\hspace{-0.8em}[A'_t-(A'_tP_{t+1}B_t+S'_t)\hat{R}_t^{\dagger}B'_t]\eta_t.
\end{eqnarray}
By backward iteration, we can derive that $\eta_t=0$ for $t\in \mathbb{N}$ with the terminal value $\eta_{N-1}=0$, which means that $v^*$ is exactly (\ref{9-3}).

``Sufficiency''. For any $  u\in L^2(\mathbb{N}; \mathbb{R}^m)$,
from (\ref{1}) and (\ref{7}), we can derive
\begin{eqnarray*}
&\hspace{-0.8em}&\hspace{-0.8em}\sum\limits^{N-1}_{t=0}(x'_tP_tx_t-x'_{t+1}P_{t+1}x_{t+1})\nonumber\\
=&\hspace{-0.8em}&\hspace{-0.8em} x'_0P_0x_0-x'_{N}Hx_{N}.
\end{eqnarray*}
Adding it into $J(x_0, u)$, we have
\begin{eqnarray}\label{16}
  &\hspace{-0.8em}&\hspace{-0.8em}J(x_0, u)-x'_0P_0x_0 \nonumber\\
=&\hspace{-0.8em}&\hspace{-0.8em}\sum\limits^{N-1}_{t=0}[x'_t(-P_t+Q_t+A'_tP_{t+1}A_t)x_t
\nonumber\\
&\hspace{-0.8em}&\hspace{-0.8em}+2x'_t(A'_tP_{t+1}B_t+S'_t)u_t
+u'_t\hat{R}_tu_t]\nonumber\\
=&\hspace{-0.8em}&\hspace{-0.8em}\sum\limits^{N-1}_{t=0}[x'_t(K^*_t)'\hat{R}_tK^*_tx_t
-2x'_t(K^*_t)' \hat{R}_tu_t
+u'_t\hat{R}_tu_t]\nonumber\\
&\hspace{-0.8em}&\hspace{-0.8em}\nonumber\\
=&\hspace{-0.8em}&\hspace{-0.8em}\sum\limits^{N-1}_{t=0}(u_t-K^*_tx_t)'\hat{R}_t
(u_t-K^*_tx_t).
\end{eqnarray}
Based on the discussion above, it can be concluded that
\begin{eqnarray*}
J(x_0, K^*x^*+v^*) \leq J(x_0, Kx+v)
\end{eqnarray*}
if and only if $\hat{R}_t \geq 0$.  That is, $(K^*, v^*)$ is the closed-loop  optimal solution to  Problem (LQ). This completes the proof.
\end{proof}

\begin{remark}
Theorem \ref{theorem1} has shown that the Problem (LQ) is closed-loop solvable if and only if the generalized Riccati equation (\ref{7}) admits a regular solution. Thus, in Example \ref{example1}, the closed-loop solution calculated from Riccati equation (\ref{7}) cannot be found due to
\begin{eqnarray}
 Range( R+B'P_{t+1}B)&\hspace{-0.8em}=&\hspace{-0.8em}Range( 0)=\{0\},\\
 Range(B'P_{t+1}A+S)&\hspace{-0.8em}=&\hspace{-0.8em}Range(1)=\mathbb{R}.
\end{eqnarray}
That is,
\begin{eqnarray*}
Range(B'P_{t+1}A+S)  \nsubseteq  Range( R+B'P_{t+1}B).
\end{eqnarray*}
\end{remark}

For convenience in future discussions, the relationship between the convexity of the map $u_t\mapsto J(x_0, u_t)$ and the regular solvability of the Riccati equation (\ref{7}) is given as follows, which is obtained based on the special case of $x_0=0$.
\begin{lemma}\label{theorem2}
Under Assumption \ref{A1}-\ref{A2}, there are following statements.
(1) Suppose Problem (LQ) is open-loop solvable, then, for any $u\in L^2(\mathbb{N}; \mathbb{R}^m)$, $J(0, u)\geq 0$.

(2) If $J(0, u)\geq 0$, then,  the map $u\mapsto J(x_0, u)$ is convex.

(3) Suppose that  there exists a constant $\alpha >0$ such that
\begin{eqnarray*}
 J(0, u)\geq \alpha \sum\limits^{N-1}_{t=0}|u_t|^2, \quad \forall u_t\in L^2(\mathbb{N}; \mathbb{R}^m).
\end{eqnarray*}
Then, the Riccati equation (\ref{7}) has a unique solution $P_t$ for $t\in\mathbb{N}$ such that
$\hat{R}_t \geq \alpha I.$
To this end, Problem (LQ) is uniquely closed-loop solvable, and it's uniquely open-loop solvable. The unique closed-loop optimal solution is given by
\begin{eqnarray*}
K^*_t&\hspace{-0.8em}=&\hspace{-0.8em} - \hat{R}_t^{-1}(B'_tP_{t+1}A_t+S_t),\\
v^*_t&\hspace{-0.8em}=&\hspace{-0.8em}0.
\end{eqnarray*}
Accordingly, the open-loop optimal control of Problem (LQ) with initial value $x_0$ is such that $u^*_t=K^*_tx^*_t$ for $t\in\mathbb{N}$, where $x^*_t$ satisfies the closed-loop system (\ref{11-1}) with initial value $x^*_0=\bar{x}_0=x_0$.
\end{lemma}

\begin{proof}
We can prove (1) and (2)  using the first and second-order directional derivatives of $J(x_0, u)$ in Lemma 2.1 and Theorem 2.1 of \cite{Ni2015}.  Applying the results obtained from Proposition 3.2 in \cite{Wang2021}, we can prove (3).
\end{proof}

\section{Open-Loop Solvability by Perturbation Method}
For any $\varepsilon >0$, the perturbed cost functional is given by
\begin{eqnarray}\label{18}
J_{\varepsilon}(x_0, u)&\hspace{-0.8em}=&\hspace{-0.8em}J(x_0, u)+\varepsilon\sum\limits^{N-1}_{t=0}|u_t|^2\nonumber\\
&\hspace{-0.8em}=&\hspace{-0.8em} \sum\limits^{N-1}_{t=0}[x'_tQ_tx_t+2u'_tS_tx_t+u'_t(R_t+\varepsilon I)u_t]\nonumber\\
&\hspace{-0.8em}&\hspace{-0.8em}+x'_{N}Hx_{N},
\end{eqnarray}
where $x$ is the state subject to (\ref{1}). The problem with the perturbed cost functional is given below.
\smallskip

\emph{Problem (LQ)$_{\varepsilon}$.}
For any initial value $x_0$, find the control $\bar{u}\in L^2(\mathbb{N}; \mathbb{R}^m)$ satisfying
\begin{eqnarray*}
 J_{\varepsilon}(x_0, \bar{u})\leq J_{\varepsilon}(x_0, u), \quad \forall u\in L^2(\mathbb{N}; \mathbb{R}^m).
\end{eqnarray*}

Moreover, denote the value function of the Problem  (LQ)$_{\varepsilon}$ as
$V_{\varepsilon}(x_0)=\inf_{ u\in L^2(\mathbb{N}; \mathbb{R}^m)} J_{\varepsilon}(x_0, u_t)$.

\begin{theorem}\label{thm3}

Problem (LQ)$_{\varepsilon}$ is unique closed-loop solvable with $(K^{\varepsilon}_t, v^{\varepsilon}_t)$ for $t\in\mathbb{N}$ satisfying
\begin{eqnarray}
 K^{\varepsilon}_t&\hspace{-0.8em}=&\hspace{-0.8em}- (R_t+B'_tP^{\varepsilon}_{t+1}B_t+\varepsilon I)^{-1}(B'_tP^{\varepsilon}_{t+1}A_t+S_t),\label{19-1}\\
 v^{\varepsilon}_t&\hspace{-0.8em}=&\hspace{-0.8em}0.\label{19-2}
\end{eqnarray}
Accordingly, for the initial value $x_0$, the open-loop solution to Problem (LQ)$_{\varepsilon}$ is given by
\begin{eqnarray}\label{20}
 u^{\varepsilon}_t &\hspace{-0.8em}=&\hspace{-0.8em} K^{\varepsilon}_tx^{\varepsilon}_t,\quad t\in\mathbb{N},
\end{eqnarray}
where $x^{\varepsilon}_t$ is the solution to the closed-loop system
\begin{eqnarray}\label{21}
 x^{\varepsilon}_{t+1} &\hspace{-0.8em}=&\hspace{-0.8em} (A_t+B_tK^{\varepsilon}_t)x^{\varepsilon}_t, \quad
 x^{\varepsilon}_0=x_0,
\end{eqnarray}
and $P^{\varepsilon}_{t}$ is the solution to the Riccati equation
\begin{eqnarray}\label{22}
P^{\varepsilon}_t&\hspace{-0.8em}=&\hspace{-0.8em} Q_t+A'_tP^{\varepsilon}_{t+1}A_t-(A'_tP^{\varepsilon}_{t+1}B_t+S'_t)
(R_t+B'_tP^{\varepsilon}_{t+1}\nonumber\\
&\hspace{-0.8em}&\hspace{-0.8em}\times B_t+\varepsilon I)^{-1}(B'_tP^{\varepsilon}_{t+1}A_t+S_t),
\end{eqnarray}
such that $R_t+B'_tP^{\varepsilon}_{t+1}B_t \geq \varepsilon I$.
\end{theorem}

\begin{proof}
Following from (\ref{18}), one has
\begin{eqnarray*}
J_{\varepsilon}(0, u)\geq \varepsilon\sum\limits^{N-1}_{t=0}|u_t|^2.
\end{eqnarray*}
Then, based on the proof of Lemma \ref{theorem2}, we obtain the desired results (\ref{19-1})-(\ref{22}).
\end{proof}

Correspondingly, the relationship between the value functions is given.
\begin{lemma}\label{lemma3}
Under the Assumption \ref{A1}-\ref{A2}, considering the system (\ref{1}), (\ref{2}) and (\ref{18}) with the initial value $x_0$,
there holds
\begin{eqnarray*}
  \lim_{\varepsilon \rightarrow 0} V_{\varepsilon}(x_0)=V(x_0).
\end{eqnarray*}
\end{lemma}

\begin{proof}
From (\ref{18}) and $\varepsilon >0$, we get
\begin{eqnarray*}
J_{\varepsilon}(x_0, u) \geq J(x_0, u)\geq V(x_0).
\end{eqnarray*}
The arbitrariness of $u\in L^2(\mathbb{N}; \mathbb{R}^m)$ means that $V_{\varepsilon}(x_0)\geq V(x_0)$.

For any $\delta >0$, which is unrelated with $\varepsilon$, there exists a controller $u^{\delta}\in L^2(\mathbb{N}; \mathbb{R}^m)$ satisfying
\begin{eqnarray*}
J(x_0, u^{\delta}) \leq V(x_0)+\delta.
\end{eqnarray*}
Using cost functional (\ref{18}), we get
$V_{\varepsilon}(x_0)\leq J(x_0, u^{\delta})+\varepsilon\sum\limits^{N-1}_{t=0}|u^{\delta}_t|^2\leq V(x_0)+\delta+\varepsilon\sum\limits^{N-1}_{t=0}|u^{\delta}_t|^2.$
Taking $\varepsilon \rightarrow 0$, the above inequality reduces to
\begin{eqnarray*}
V(x_0)\leq V_{\varepsilon}(x_0)\leq V(x_0)+\delta.
\end{eqnarray*}
The arbitrariness of $\delta >0$ means that $\lim_{\varepsilon \rightarrow 0} V_{\varepsilon}(x_0)=V(x_0)$. This completes the proof.
\end{proof}

To this end, the main results of this section is given, which reveals the relationship between the open-loop solution of Problem (LQ)  and $\{u^{\varepsilon}\}$.
\begin{theorem}\label{theorem3}
Under the Assumption \ref{A1}-\ref{A2}, for any initial value $x_0$, let $u^{\varepsilon}$ be defined in (\ref{20}), which is derived from the closed-loop optimal solution ($K^{\varepsilon}, v^{\varepsilon}$) for Problem (LQ)$_{\varepsilon}$. The following conditions are equivalent:

(1) Problem (LQ)  is open-loop solvable with the initial value $x_0$;

(2) The sequence $\{u^{\varepsilon}\}$ is bounded in the Hilbert space $L^2(\mathbb{N}; \mathbb{R}^m)$, that is to say $\sup_{\varepsilon >0}\sum\limits^{N-1}_{t=0}|u^{\varepsilon}_t|^2 < \infty;$

(3) The sequence $\{u^{\varepsilon}\}$ converges strongly in $L^2(\mathbb{N}; \mathbb{R}^m)$ when $\varepsilon \rightarrow 0$.

\end{theorem}

\begin{proof}
(1) $\Rightarrow$ (2). Since $u^{\varepsilon}$ is the open-loop optimal control for Problem (LQ)$_{\varepsilon}$, combining with (\ref{18}), one has
$V_{\varepsilon}(x_0)=J_{\varepsilon}(x_0, u^{\varepsilon})=J(x_0, u^{\varepsilon})+\varepsilon\sum\limits^{N-1}_{t=0}|u^{\varepsilon}_t|^2
 \geq V(x_0)+\varepsilon\sum\limits^{N-1}_{t=0}|u^{\varepsilon}_t|^2.$
Moreover, suppose that $\bar{u}$ is the open-loop optimal solution to Problem (LQ) with the initial value $x_0$. Based on cost functional (\ref{18}), we obtain
$V_{\varepsilon}(x_0)\leq J_{\varepsilon}(x_0, \bar{u})=J(x_0, \bar{u})+\varepsilon\sum\limits^{N-1}_{t=0}|\bar{u}_t|^2
 = V(x_0)+\varepsilon\sum\limits^{N-1}_{t=0}|\bar{u}_t|^2.$
From the discussion above, for $\varepsilon >0$, we have
\begin{eqnarray}\label{23}
 \sum\limits^{N-1}_{t=0}|u^{\varepsilon}_t|^2 \leq
 \frac{1}{\varepsilon}(V_{\varepsilon}(x_0)-V(x_0))\leq \sum\limits^{N-1}_{t=0}|\bar{u}_t|^2,
\end{eqnarray}
which implies that $\sup_{\varepsilon >0}\sum\limits^{N-1}_{t=0}|u^{\varepsilon}_t|^2 < \infty$.

(2) $\Rightarrow$ (1). The boundedness of $\{u^{\varepsilon}\}$ means that there exists a sequence $\{\varepsilon_k\}$ ($k\in\{1, 2, \cdots, \infty\} $) and $u^*\in L^2(\mathbb{N}; \mathbb{R}^m)$ satisfying $\lim_{k \rightarrow \infty}\varepsilon_k=0$ and  $\{u^{\varepsilon_k}\}$ is convergent weakly to a $u^*$, i.e.,
\begin{eqnarray*}
  \|u^*_t\|\leq \lim_{k \rightarrow \infty}\inf \|u^{\varepsilon_k}_t\|.
\end{eqnarray*}
Since $J(x_0, u)$ is convex and continuous with respect to $u$, then there holds
\begin{eqnarray*}
J(x_0, u^*)&\hspace{-0.8em}\leq&\hspace{-0.8em} \lim_{k \rightarrow \infty}\inf J(x_0, u^{\varepsilon_k})\nonumber\\
&\hspace{-0.8em}=&\hspace{-0.8em}\lim_{k \rightarrow \infty}\inf[J_{\varepsilon_k}(x_0, u^{\varepsilon_k})-\varepsilon_k\sum\limits^{N-1}_{t=0}|u^{\varepsilon_k}_t|^2]\nonumber\\
&\hspace{-0.8em}=&\hspace{-0.8em}\lim_{k \rightarrow \infty}  V_{\varepsilon_k}( x_0)=V(x_0).
\end{eqnarray*}
Hence, it's concluded that $u^*$ is an open-loop optimal control for  Problem (LQ).

(3) $\Rightarrow$ (2). This part of the proof is obvious.

(2) $\Rightarrow$ (3). Suppose $\{u^{\varepsilon}\}$ converges strongly in $L^2(\mathbb{N}; \mathbb{R}^m)$ when $\varepsilon \rightarrow 0$, and following from the previous proof, there exists  $\{u^{\varepsilon_k}\}$ converges weakly to an open-loop optimal control.
In this part, firstly, we will show the sequence $\{u^{\varepsilon}\}$ is convergent weakly to the same open-loop optimal control for Problem (LQ) with initial value $x_0$ as $\varepsilon\rightarrow 0$. That is, if there exist two sequences $\{u^{\varepsilon_{i,k}}\}\subseteq \{u^{\varepsilon}\}$ ($i=1, 2$),
which converge weakly to $u^{i, *}$, then $u^{1, *}=u^{2, *}$ is the open-loop optimal
control for Problem (LQ).

It's obvious that $u^{1, *}$ and $u^{2, *}$ are both the open-loop optimal solution to
Problem (LQ) with the initial value $x_0$. Since $J(x_0, u)$ is convex with respect to $u$, there holds that
$J(x_0, \frac{1}{2}(u^{1, *}+u^{2, *}))  \leq  \frac{1}{2} J(x_0, u^{1, *})
+     \frac{1}{2} J(x_0, u^{2, *})= V(x_0).$
That is, $\frac{1}{2}(u^{1, *}+u^{2, *})$ is also an open-loop optimal solution with
the initial value  $x_0$.  Using (\ref{23}), it yields
\begin{eqnarray}
 \sum\limits^{N-1}_{t=0}|u^{\varepsilon_{i,k}}_t|^2 \leq
 \sum\limits^{N-1}_{t=0}|\frac{1}{2}(u^{1, *}_t+u^{2, *}_t)|^2, \quad i=1, 2.
\end{eqnarray}
Due to $\|u^{i, *}_t\|\leq \lim_{k \rightarrow \infty}\inf \|u^{\varepsilon_{i,k}}_t\|$, the above inequality reduces to the following form
\begin{eqnarray}
 \sum\limits^{N-1}_{t=0}|u^{i, *}_t|^2 \leq
 \sum\limits^{N-1}_{t=0}|\frac{1}{2}(u^{1, *}_t+u^{2, *}_t)|^2, \quad i=1, 2,
\end{eqnarray}
which generates
$2[\sum\limits^{N-1}_{t=0}(u^{1, *}_t)^2+\sum\limits^{N-1}_{t=0}(u^{2, *}_t)^2]\leq
\sum\limits^{N-1}_{t=0}(u^{1, *}_t+u^{2, *}_t)^2,$
i.e., $\sum\limits^{N-1}_{t=0}(u^{1, *}_t-u^{2, *}_t)^2\leq 0$. It holds if and only if $u^{1, *}=u^{2, *}$. That is  to say, the sequence $\{u^{\varepsilon}\}$ converges weakly to the same open-loop optimal control of Problem (LQ) along with the initial value $x_0$ as  $\varepsilon \rightarrow 0$.

Next, the sequence $\{u^{\varepsilon}\}$ converges strongly in $L^2(\mathbb{N}; \mathbb{R}^m)$ when $\varepsilon \rightarrow 0$ will be proved.
Suppose $u^*$ is an open-loop optimal control of Problem (LQ), then $\{u^{\varepsilon}\}$ converges weakly to $u^*$, i.e., $\|u^*_t\|\leq \lim_{\varepsilon \rightarrow 0}\inf \|u^{\varepsilon}_t\|$, which means that
$\sum\limits^{N-1}_{t=0}|u^*_t|^2\leq \lim_{\varepsilon \rightarrow 0}\inf \sum\limits^{N-1}_{t=0}|u^{\varepsilon}_t|^2.$
Moreover, from (\ref{23}) there always holds
$\sum\limits^{N-1}_{t=0}|u^{\varepsilon}_t|^2 \leq \sum\limits^{N-1}_{t=0}|u^*_t|^2.$
Combining with the above equation, we have
$\lim_{\varepsilon \rightarrow 0}  \sum\limits^{N-1}_{t=0}|u^{\varepsilon}_t|^2
=\sum\limits^{N-1}_{t=0}|u^*_t|^2.$

Based on the discussion above, it yields
$\lim_{\varepsilon \rightarrow 0}\sum\limits^{N-1}_{t=0}(u^{\varepsilon}_t-u^*_t)^2
=\lim_{\varepsilon \rightarrow 0}[\sum\limits^{N-1}_{t=0}|u^{\varepsilon}_t|^2+\sum\limits^{N-1}_{t=0}|u^*_t|^2
-2\sum\limits^{N-1}_{t=0}(u^*_t)'u^{\varepsilon}_t]=0.$
That is to say, $\{u^{\varepsilon}\}$ converges strongly to $u^*$ as $\varepsilon \rightarrow 0$.
\end{proof}

\begin{corollary}
When any one of the conditions (1) to (3) in Theorem \ref{theorem3} is met, the sequence $\{u^{\varepsilon}\}$ is convergent strongly to an open-loop optimal control of Problem (LQ) along with the initial value $x_0$ as  $\varepsilon \rightarrow 0$.
\end{corollary}

\begin{proof}
The proof can be obtained by the proof of (2) $\Rightarrow$ (1) and (2) $\Rightarrow$ (3) in Theorem \ref{theorem3}.
\end{proof}

\section{Weak Closed-Loop Solvability of Problem (LQ)}
It has been shown that if
\begin{eqnarray*}
Range(B'_tP_{t+1}A_t+S_t)  \nsubseteq  Range(\hat{R}_t)£¬
\end{eqnarray*}
then the gain matrix $K$ calculated in (\ref{9-1}) by the Riccati equation (\ref{7}) will no-longer be the closed-loop optimal solution of Problem (LQ). In this section, the relation with the open-loop and weak closed-loop solution to Problem (LQ) is revealed.

Firstly, we show that $K^{\varepsilon}$ defined in (\ref{19-1}) is convergent locally in $\tilde{\mathbb{N}}$, where $\tilde{\mathbb{N}}\subseteq {\mathbb{N}}$ with $\tilde{\mathbb{N}}=\{0, 1, \cdots, m\}$, $m\leq N-2$.
\begin{lemma}\label{lemma4}
Under the Assumption \ref{A1}-\ref{A2}, assume Problem (LQ) is open-loop solvable. Then, for any $\tilde{\mathbb{N}}\subseteq {\mathbb{N}}$, the set $\{K^{\varepsilon}\}$ defined in (\ref{19-1}) is convergent in $L^2(\tilde{\mathbb{N}}; \mathbb{R}^{m\times n})$, i.e.,
\begin{eqnarray}\label{28}
\lim_{\varepsilon \rightarrow 0}\sum\limits_{t\in \tilde{\mathbb{N} }}|K^{\varepsilon}_t-K^*_t|^2=0, \quad  \tilde{\mathbb{N}}\subseteq \bar{\mathbb{N}},
\end{eqnarray}
where $K^*\in L^2(\tilde{\mathbb{N}}; \mathbb{R}^{m\times n})$ is deterministic.
\end{lemma}

\begin{proof}
Denote $H_{\varepsilon}(t, l)=\prod\limits^{t}_{i=l}(A_i+B_iK^{\varepsilon}_i)$ for $t\geq l$ while $H_{\varepsilon}(t,l)=I$ for $t< l$.
Following from (\ref{21}), $x^{\varepsilon}_t$ can be recalculated as
$x^{\varepsilon}_t=H_{\varepsilon}(t-1, 0)x_0$
with $x^{\varepsilon}_0=x_0$.

According to Theorem \ref{theorem3}, the open-loop solvability of Problem (LQ) means that
$u^{\varepsilon}$ defined in (\ref{20}) is strongly convergent, that is,
$u^{\varepsilon}_t =K^{\varepsilon}_tx^{\varepsilon}_t=K^{\varepsilon}_tH_{\varepsilon}(t-1, 0)x_0$
is convergent strongly in $L^2(\mathbb{N}; \mathbb{R}^m)$ for any $\varepsilon >0$.
Denoting $\Lambda^{\varepsilon}_t=K^{\varepsilon}_tH_{\varepsilon}(t-1, 0)$, it follows that $\Lambda^{\varepsilon}$ converges strongly in $L^2(\mathbb{N}; \mathbb{R}^{m\times n})$, and the strong limits is denoted as $\Lambda^{*}$, which indicates that
\begin{eqnarray}\label{26}
\lim_{\varepsilon \rightarrow 0 }\sum\limits^{N-1}_{t=0}|\Lambda^{\varepsilon}_t-\Lambda^{*}_t|^2=0.
\end{eqnarray}

Next, the convergence of $H_{\varepsilon}(t, 0)$ for $t\in\mathbb{N}$ is proved. By the definition, $H_{\varepsilon}(t, 0)$ is rewritten as
\begin{eqnarray*}
H_{\varepsilon}(t, 0)&\hspace{-0.8em}=&\hspace{-0.8em}(A_t+B_tK^{\varepsilon}_t)H_{\varepsilon}(t-1, 0)\nonumber\\
&\hspace{-0.8em}=&\hspace{-0.8em}...=\Big(\prod\limits^{t}_{i=0}A_i\Big) H_{\varepsilon}(-1, 0)+\sum\limits^{t}_{i=0}\Big(\prod\limits^{t}_{l=i+1}A_l\Big) B_i\Lambda^{\varepsilon}_i,
\end{eqnarray*}
with the initial value $H_{\varepsilon}(-1, 0)=I$. Thus, $H_{\varepsilon}(t, 0)$ for $t\in\mathbb{N}$ is convergent since $\Lambda^{\varepsilon}_t$ converges strongly to $\Lambda^{*}_t$ on $t\in\mathbb{N}$ and assume the convergent function is $H^*_{\varepsilon}(t, 0)$ with $t\in\mathbb{N}$ while $H^*_{\varepsilon}(-1, 0)=I$.

By using $H_{\varepsilon}(-1, 0)=I$, there exists a small constant $\Delta_s>0$, for
$t\in \{s, s+1, \cdots, s+\Delta_s\}$, satisfying $H_{\varepsilon}(t-1, 0)$ is invertible and $|H_{\varepsilon}(t-1, 0)|^2\geq \frac{1}{3}$.
Then, we have
\begin{eqnarray}\label{27}
&\hspace{-0.8em}&\hspace{-0.8em}\sum\limits^{s+\Delta_s}_{k=s}|K^{\varepsilon_1}_t-K^{\varepsilon_2}_t|^2\nonumber\\
=&\hspace{-0.8em}&\hspace{-0.8em}\sum\limits^{s+\Delta_s}_{k=s}
|\Lambda^{\varepsilon_1}_tH^{-1}_{\varepsilon_1}(t-1, 0)
-\Lambda^{\varepsilon_2}_tH^{-1}_{\varepsilon_1}(t-1, 0)\nonumber\\
&\hspace{-0.8em}&\hspace{-0.8em}+\Lambda^{\varepsilon_2}_tH^{-1}_{\varepsilon_1}(t-1, 0)-\Lambda^{\varepsilon_2}_tH^{-1}_{\varepsilon_2}(t-1, 0)|^2\nonumber\\
\leq &\hspace{-0.8em}&\hspace{-0.8em} 18\sum\limits^{s+\Delta_s}_{k=s}|\Lambda^{\varepsilon_1}_t-\Lambda^{\varepsilon_2}_t|^2
+162\sum\limits^{s+\Delta_s}_{k=s}|\Lambda^{\varepsilon_2}_t|^2\nonumber\\
&\hspace{-0.8em}&\hspace{-0.8em}\times
\sup_{t\in[s, s+\Delta_t]}|H_{\varepsilon_1}(t-1, 0)-H_{\varepsilon_2}(t-1, 0)|^2.
\end{eqnarray}
Due to (\ref{26}) and the convergence of $H_{\varepsilon}(t, 0)$ on $t\in\mathbb{N}$, by letting $\varepsilon_1, \varepsilon_2 \rightarrow 0$, (\ref{27}) is reduced into
$\lim_{\varepsilon_1, \varepsilon_2 \rightarrow 0}\sum\limits^{s+\Delta_s}_{k=s}|K^{\varepsilon_1}_t-K^{\varepsilon_2}_t|^2=0.$

Due to the compactness of $\tilde{\mathbb{N}}$, finite numbers $s_1$, ... $s_r$ can be selected such that $\tilde{\mathbb{N}}\subseteq \bigcup\limits^{r}_{j=1}[s_j, s_j+\Delta_{s_j}]$ and $\lim_{\varepsilon_1, \varepsilon_2\rightarrow 0} \sum\limits^{s_j+\Delta_{s_j}}_{k=s_j}|K^{\varepsilon_1}_t-K^{\varepsilon_2}_t|^2=0$.

Then, when $\varepsilon_1, \varepsilon_2 \rightarrow 0$ it derives that
$\sum\limits_{t\in\tilde{\mathbb{N}}}|K^{\varepsilon_1}_t-K^{\varepsilon_2}_t|^2\leq
\sum\limits^{r}_{j=1}\sum\limits^{s_j+\Delta_{s_j}}_{t=s_j}
|K^{\varepsilon_1}_t-K^{\varepsilon_2}_t|^2\longrightarrow 0.$
This completes the proof.
\end{proof}

In the following, the equivalent relationship between the open-loop and weak closed-loop solvability for Problem (LQ) is estabilished.

\begin{theorem}\label{theorem5}
Under the Assumption \ref{A1}-\ref{A2}, if Problem (LQ) is open-loop solvable, then $(K^*, 0)$ derived in Lemma \ref{lemma4} is a weak closed-loop optimal solution for Problem (LQ) with $t\in \mathbb{N}$.
Accordingly, the open-loop solvability of Problem (LQ) is equivalent to the weak closed-loop solvability of Problem (LQ).
\end{theorem}

\begin{proof}
Due to the open-loop solvability of Problem (LQ), and following from Theorem \ref{theorem3},
$u^{\varepsilon}$  defined in (\ref{20}) is strongly convergent to an open-loop optimal control of Problem (LQ), which is denoted $u^*$. The corresponding optimal state, which is called $x^*$ is such that $x^*_{t+1} = A_tx^*_t+B_tu^*_t$ with $x^*_0=x_0$.
%
%

Next, we will prove that $u^*_t=K^*_tx^*_t$ for $t\in \tilde{\mathbb{N}}$, where $K^*_t$ is deterministic and satisfies (\ref{28}). Then, $(K^*, 0)$ is a weak closed-loop optimal solution to Problem (LQ). Firstly, we show that
\begin{eqnarray}\label{29}
\sup_{t\in \mathbb{N}}|x_t|^2 \leq L (|x_0|^2+ \sum^{N-1}_{t=0} |u_t|^2),
\end{eqnarray}
where $x$, $u$ are subject to (\ref{1}) and $L>0$ represents a genetic constant which can be different from line to line. By iteration, (\ref{1}) is rewritten as 
$x_t=\Big(\prod\limits^{t-1}_{i=0}A_i\Big) x_0+\sum\limits^{t-1}_{i=0}\Big(\prod\limits^{t-1}_{l=i+1}A_l\Big) B_iu_i.$
Thus,  a constant $L>0$ which is different from line to line can be selected  such that
$\sup_{t\in \mathbb{N}}|x_t|^2\leq L |x_0+\sum\limits^{t-1}_{i=0}u_i|^2\leq
L(|x_0|^2+\sum\limits^{N-1}_{i=0}|u_i|^2)$.
According to (\ref{29}), it follows that 
$\sup_{t\in \mathbb{N}}|x^{\varepsilon}_t-x^*_t|^2 \leq
L \sum^{N-1}_{t=0} |u^{\varepsilon}_t-u^*_t|^2$
where $x^{\varepsilon}$ is the solution to (\ref{21}). Combining above equation with Theorem \ref{theorem3}, one has
\begin{eqnarray}\label{31}
\lim_{\varepsilon\rightarrow 0}\sup_{t\in \mathbb{N}}|x^{\varepsilon}_t-x^*_t|^2=0.
\end{eqnarray}

Now, we are in position to prove $u^*_t=K^*_tx^*_t$ ($t\in \tilde{\mathbb{N}}$). Considering
$\sum_{t\in\tilde{\mathbb{N}}}  |u^{\varepsilon}_t-K^*_tx^*_t|^2\leq 2\sum_{t\in\tilde{\mathbb{N}}} |K^{\varepsilon}_t|^2\sup_{t\in \mathbb{N}} |x^{\varepsilon}_t-x^*_t|^2
+2\sum_{t\in\tilde{\mathbb{N}}} |K^{\varepsilon}_t-K^*_t|^2  \sup_{t\in \mathbb{N}} |x^*_t|^2$
and according to Lemma \ref{lemma4} and (\ref{31}), one has 
$\lim_{\varepsilon\rightarrow 0} \sum_{t\in\tilde{\mathbb{N}}}  |u^{\varepsilon}_t-K^*_tx^*_t|^2=0.$
As stated in Theorem \ref{theorem3}, $u^{\varepsilon}_t=K^{\varepsilon}_tx^{\varepsilon}_t$  converges strongly to $u^*_t$ for $t\in \mathbb{N}$  as $\varepsilon\rightarrow 0$, which means $u^*_t=K^*_tx^*_t$ ($t\in \tilde{\mathbb{N}}$) is established. That is to say, the open-loop solvability indicates the weak close-loop solvability. Moreover, using the definition of weak close-loop solvability, the weak close-loop solvability of Problem (LQ) must be open-loop solvable. The proof is completed.

\end{proof}

\section{Numerical Example}
Consider the  system and associated cost functional
\begin{eqnarray}
x_{t+1} &\hspace{-0.8em}=&\hspace{-0.8em} x_t+u_t, \quad t\in\{0, 1\},\label{11}\\
J(x_0, u) &\hspace{-0.8em}=&\hspace{-0.8em} x^2_2-\sum\limits^{1}_{t=0}u^2_t,\label{12}
\end{eqnarray}
with initial value $x_0$. 
Following from the difference Riccati equation (\ref{7})
with the terminal value $P_2=1$, and by simple calculation, the solution of it is exactly $P_t\equiv 1$, $t\in \{0, 1, 2\}$.

From Theorem \ref{theorem1}, it should be noted that the problem with system (\ref{11})-(\ref{12}) is open-loop solvable but not closed-loop solvable due to the following reason:
\begin{eqnarray*}
 Range( R+B'P_{t+1}B)&\hspace{-0.8em}=&\hspace{-0.8em}Range( 0)=\{0\},\\
 Range(B'P_{t+1}A+S)&\hspace{-0.8em}=&\hspace{-0.8em}Range(1)=\mathbb{R}.
\end{eqnarray*}
That is, $Range(B'P_{t+1}A+S)  \nsubseteq  Range( R+B'P_{t+1}B).$

Taking $\varepsilon>0$ and by Theorem \ref{thm3}, the solution to Riccati equation (\ref{22}) is calculated as $P^{\varepsilon}_t=\frac{P^{\varepsilon}_{t+1}(\varepsilon-1)}{P^{\varepsilon}_{t+1}+\varepsilon-1}$,
%
%
with the terminal value $P^{\varepsilon}_2=1$. By backward iteration, $P^{\varepsilon}_t$ is actually calculated as $P^{\varepsilon}_t=\frac{\varepsilon-1}{\varepsilon+(1-t)}$.

Accordingly, the closed-loop optimal gain matrix is such that
$K^{\varepsilon}_t=-\frac{P^{\varepsilon}_t}{\varepsilon-1}=-\frac{1}{\varepsilon+(1-t)}$.
Hence, the corresponding closed-loop state equation is written as
$x^{\varepsilon}_{t+1}=(1+K^{\varepsilon}_t)x^{\varepsilon}_t$ with $x^{\varepsilon}_0=x_0$.
The control is given by $u^{\varepsilon}_t=K^{\varepsilon}_tx^{\varepsilon}_t=K^{\varepsilon}_t
\prod\limits^{t-1}_{i=0}(1+K^{\varepsilon}_i)x^{\varepsilon}_0$,
that is to say,
\begin{eqnarray*}
u^{\varepsilon}_0 &\hspace{-0.8em}=&\hspace{-0.8em} K^{\varepsilon}_0 x^{\varepsilon}_0 =-\frac{1}{\varepsilon+1}x_0,\\
u^{\varepsilon}_1 &\hspace{-0.8em}=&\hspace{-0.8em} K^{\varepsilon}_1 x^{\varepsilon}_1
=K^{\varepsilon}_1(1+K^{\varepsilon}_0)x^{\varepsilon}_0=-\frac{1}{\varepsilon+1}x_0.
\end{eqnarray*}

Since $\sum\limits^{1}_{t=0}|u^{\varepsilon}_t|^2 = \frac{2x^2_0}{(\varepsilon+1)^2}\leq  2x^2_0$
that is to say, $u^{\varepsilon}_t$ is bounded. From the results in Theorem \ref{theorem3}, the open-loop optimal solution is given as $u^{*}_0 = \lim_{\varepsilon\rightarrow 0}u^{\varepsilon}_0=-x_0$ and $u^{*}_1 = \lim_{\varepsilon\rightarrow 0}u^{\varepsilon}_1=-x_0$.

Finally, based on the discussion in Theorem \ref{theorem5}, the weak closed-loop optimal strategy $K^{*}_t$ is given by
\begin{eqnarray*}
K^{*}_t &\hspace{-0.8em}=&\hspace{-0.8em} \lim_{\varepsilon\rightarrow 0}K^{\varepsilon}_t
=\lim_{\varepsilon\rightarrow 0}(-\frac{1}{\varepsilon+(1-t)})
=-\frac{1}{(1-t)}.
\end{eqnarray*}

\section{Conclusion}

This paper discusses the open-loop, closed-loop, and weak closed-loop solvability for the deterministic discrete-time system. For this kind of LQ control problem, the closed-loop solvability, equivalent to the regular solvability of the generalized Riccati equation, indicates the open-loop solvability but not vice versa. When the LQ problem is merely open-loop solvable, we have found a weak closed-loop solution whose outcome is the open-loop solution to the LQ problem so that a linear feedback form of the state exists to the open-loop solution. Moreover, an equivalent relationship between the open-loop and the weak closed-loop solutions to the LQ problem has been established.

\section*{Acknowledgments}
This research is supported by the National Natural Science Foundation of China (No. 12201129)
and Natural Science Foundation of Guangdong Province
(No. 2022A1515010839).

\vfill

\end{document}